\documentclass{article}

\usepackage{amsmath, amsthm, amssymb, bm, mathtools, xfrac}

\numberwithin{equation}{section} 

\usepackage[utf8]{inputenc}
\usepackage[T1]{fontenc}

\usepackage{enumerate, enumitem}
\usepackage{lmodern, microtype}
\usepackage{graphicx}
\usepackage{verbatim}

\usepackage{amsmath, amsthm, amssymb, bm, esint, mathtools, xfrac}

\usepackage[usenames,dvipsnames]{color}
\PassOptionsToPackage{usenames, dvipsnames}{color}
\usepackage{tikz}
\usetikzlibrary{shapes,positioning}

\usepackage[english, ngerman]{babel}

\usepackage[%
backend=biber,
style=alphabetic
]{biblatex}
\usepackage{csquotes}

\AtEveryBibitem{\clearlist{address}}
\AtEveryBibitem{\clearfield{doi}}
\AtEveryBibitem{\clearlist{language}}
\AtEveryBibitem{\clearfield{isbn}}
\AtEveryBibitem{\clearfield{issn}}
\AtEveryBibitem{\clearfield{pages}}
\AtEveryBibitem{\clearfield{pagetotal}}

\usepackage[%
pdfpagelabels=true,
unicode=true,
colorlinks=false,
pdfborder={0 0 0}
]{hyperref}
\usepackage{xurl}
\hypersetup{breaklinks=true}

\newtheorem{prop}[equation]{Proposition}
\newtheorem{thm}[equation]{Theorem}
\newtheorem{lem}[equation]{Lemma}
\newtheorem{cor}[equation]{Corollary}

\newtheorem{defi}[equation]{Definition}
\newtheorem{rem}[equation]{Remark}

\newtheorem{assum}[equation]{Assumption}
\newtheorem{cond}[equation]{Condition}

\renewcommand{\phi}{\varphi}			
\renewcommand{\vec}[1]{\boldsymbol{#1}}	
\renewcommand{\tilde}[1]{\widetilde{#1}}
\renewcommand{\hat}[1]{\widehat{#1}}

\newcommand{\N}{\mathbb{N}}				
\newcommand{\R}{\mathbb{R}}				
\newcommand{\Rd}{\mathbb{R}^d}			

\newcommand{\closed}[1]{\overline{#1}}				

\newcommand{\abs}[1]{\left\lvert#1\right\rvert}	  	
\renewcommand{\d}{^{\ast}}							
\newcommand{\del}{\partial}							
\newcommand{\injto}{\xhookrightarrow{}} 			
\newcommand{\fp}[2]{\langle #1, #2 \rangle}      	
\newcommand{\norm}[1]{\left\lVert#1\right\rVert}  	
\newcommand{\bignorm}[1]{\big\lVert#1\big\rVert}  	
\newcommand{\supp}{\operatorname{supp}}				
\newcommand{\weakto}{\rightharpoonup}						

\newcommand{\Bog}{\mathcal{B}}						
\renewcommand{\div}{\operatorname{div}}				
\newcommand{\E}{\mathcal{E}}						

\newcommand{\D}[1]{\vec{D #1}}						
\newcommand{\tr}{\operatorname{tr}}					

\renewcommand{\S}{\boldsymbol{\mathcal{S}}}		  	

\newcommand{\intom}{\int_{\Omega}}					
\newcommand{\dotn}{\cdot \vec{n}}					
\newcommand{\dx}{\, dx}

\hypersetup{
	pdfauthor={Julius Je\ss{}berger},
	pdftitle={Weak solutions for steady, fully inhomogeneous generalized Navier-Stokes equations},
	pdflang={en} 
}

\begin{document}
\selectlanguage{english}

\begin{center}
	
	{ \Large \bfseries 
		Weak solutions for steady, fully inhomogeneous generalized Navier-Stokes equations
	}\\[0.4cm]
	{Julius~Je\ss{}berger and Michael~R\r{u}\v{z}i\v{c}ka}
	
	{\today}\\[0.4cm]
	
	{Institute of Applied Mathematics, Albert-Ludwigs-University Freiburg, Ernst-Zermelo-Str. 1, D-79104 Freiburg, Germany}
	
	{E-mail addresses: julius.jessberger@mathematik.uni-freiburg.de, rose@mathematik.uni-freiburg.de}
	
\end{center}

\textbf{Abstract}\hspace{0.15cm}
We consider the question of existence of weak solutions for the fully inhomogeneous, stationary generalized Navier-Stokes equations for homogeneous, shear-thinning fluids. For a shear rate exponent $p \in \big(\tfrac{2d}{d+1}, 2\big)$, previous results require either smallness of the norm or vanishing of the normal component of the boundary data. In this work, combining previous methods, we propose a new, more general smallness condition.

\noindent\textbf{Keywords}\hspace{0.15cm}
Generalized Newtonian fluid, existence of weak solutions, a priori estimate, inhomogeneous Dirichlet boundary condition.

\noindent\textbf{AMS Classifications (2020)}\hspace{0.15cm}
35D30, 35Q35, 76D03.

\pagestyle{plain}

\section{Introduction}

In this article, we consider the steady, generalized Navier-Stokes equations
\begin{equation} \label{eq:main_problem1}
	- \div \S(\vec{Dv}) + \div (\vec{v}\otimes\vec{v}) + \nabla \pi = \vec{f} \quad \textrm{in } \Omega
\end{equation}
for the velocity $\vec{v}$ and the pressure $\pi$, which describe the steady motion of a homogeneous, generalized Newtonian fluid.
The system is completed by inhomogeneous divergence and Dirichlet boundary conditions
\begin{equation} \label{eq:main_problem2} \begin{aligned}
		\div \vec{v} &= g_1  \quad&&\textrm{in } \Omega\,, \\
		\vec{v} &= \vec{g_2}  &&\textrm{on } \del \Omega\,.
\end{aligned} \end{equation}
Here, $\S$ is an extra stress tensor with $(p,\delta)$-structure, $\vec{f}$ is an external body force, $g_1$ and $\vec{g_2}$ are data inside/ on the boundary of a bounded Lipschitz domain $\Omega \subset \Rd$ of dimension $d \in \{2, 3\}$.
We refer to Subsection~\ref{sub:assumptions} for precise definitions of these objects.
The case $p=2$ results in the well-known, classical Navier-Stokes
equations, while for $p \neq 2$ we deal with a shear-dependent viscosity in the fluid flow equation. 

The decomposition $\vec{v} = \vec{u} + \vec{g}$, where $\vec{g}$ represents the boundary and divergence data and $\vec{u} \in W_0^{1,p}(\Omega)$ with $\div\vec{u} = 0$, yields a homogeneous, pseudomonotone problem for $\vec{u}$. Its solvability depends crucially on the growth rate $p$: in the uncritical case $p > 2$, one obtains a coercive operator and existence of weak solutions follows directly from the theory of pseudomonotone operators~\cite{r-mol-inhomo}. Due to the antisymmetry of the convective term, the same holds for the homogeneous problem for any $p > \tfrac{3d}{d+2}$~\cite{lions-quel}. This approach can be extended to the range $p > \tfrac{2d}{d+2}$ with the Lipschitz truncation method~\cite{fms2, dms}.

In the critical case $p=2$, which results in the inhomogeneous
Navier-Stokes equations, global coercivity is not clear in the general
case since it does not follow from a growth argument. Especially the
case $g_1=0$ has been studied intensively since the seminal works by
J. Leray and E. Hopf who obtained existence if the flux over each
boundary component $\del \Omega_i$ vanishes~\cite{Leray33,
  Hopf41}. Their results have been generalized in favor of a smallness
condition (cf.~\cite{Galdi} and references therein). This approach
relies on choosing an extension $\vec{g}$ such that the convective
term becomes small in the a priori estimate. Another approach, mainly
followed by M. Korobkov, K. Pileckas and R. Russo, does not require
smallness of the data but its application is, so far, restricted to a
two-dimensional or an axially symmetric geometry~\cite{KPR_2d3d}.

Until now, all results for the supercritical case $p<2$ require some 
smallness condition for the data in order to obtain local coercivity. The first result is due to E. Blavier and A. Mikeli{\'c}, where smallness of the norm~$\norm{\vec{g_2}}$ in a trace space is required and $p > \tfrac{3d}{d+2}$, $g_1=0$~\cite{mikelic}. 
M. Lanzendörfer considers pressure- and shear-rate-dependent viscosity for $g_1=0$, $p > \tfrac{2d}{d+1}$ and vanishing normal velocity component at the boundary ($\vec{g_2} \dotn = 0$). His employment of an extension $\vec{g}$ with small $L^p$-norm allows for tangential boundary data of arbitrary size~\cite{lan09,lan08} if $p >2- \tfrac{1}{d}$.
C. Sin examined electrorheological fluids with $g_1=0$ and the full range of $p > \tfrac{2d}{d+2}$. He reduces the regularity assumptions compared to~\cite{mikelic} but some of his computations could not be confirmed in~\cite{JR21_inhom}.
In \cite{JR21_inhom}, the authors consider fluids with $(p,\delta)$-structure, $p > \tfrac{2d}{d+2}$ and nonzero divergence data $g_1$. Therefore, smallness of the norms of $g_1$ and $\vec{g_2}$ is assumed and the regularity questions due to~\cite{Sin} have been partially closed.
H. Elshehabey and A. Silvestre~\cite{Elshehabey21} obtained a result
for $p > \tfrac{2d}{d+2}$, $g_1=0$ which is similar to~\cite{mikelic,
  JR21_inhom}. Moreover, they showed that it suffices that the normal
velocity component is small compared to the tangential boundary data
if $p > 2- \tfrac{1}{d}$, $g_1=0$ and $\del\Omega \in C^3$. Their
statement can be read as a perturbation result: for any non-zero size
of tangential data $\norm{\vec{g_2}\cdot \vec{t}}$, there exists a
radius $\epsilon>0$ such that if $\norm{\vec{g_2}\dotn} < \epsilon$, we find a solution.
Finally, we want to mention P.~Kaplick{\'y}'s, J.~M{\'a}lek's and J.~Star{\'a}'s article~\cite{Kaplicky07}, where strong solutions are obtained for $d=2$, $p > \tfrac{3}{2}$, $g_1=0$ whenever the norm~$\norm{\vec{g_2}}$ is small or the normal component vanishes.

The proofs in these works (apart from~\cite{Kaplicky07}) are rather similar and combine the arguments from pseudomonotone operator theory with local coercivity. However, they differ in the way how the extensions of the boundary data are constructed: while straight-forward extension operators and a divergence correction with the Bogovski\u{\i} operator are applied in \cite{mikelic, JR21_inhom}, extensions of tangential data with small $L^p$-norm are constructed in~\cite{lan09, Elshehabey21}. For $p \in \big( \tfrac{2d}{d+1}, 2\big)$, the question arises whether smallness of the normal boundary and divergence data suffices while the tangential data is arbitrary.
In contrast, for very low values $p \in \big( \tfrac{2d}{d+2}, \tfrac{2d}{d+1}\big]$, all previous results agree and it is also not clear how they could be improved in the sense of arbitrary tangential boundary data (cf.~the discussion in Section~\ref{sec:notes}).
Another difference among these works is the method of regularization if $p < \tfrac{3d}{d+2}$ which influences  the regularity requirements. In this regard, employing approximate problems via an additional elliptic term and passing to the limit via the Lipschitz truncation method seems to yield the best results. We point out that these methods are applicable to rather general settings, e.g. fluids with pressure-dependent viscosities and electrorheological fluids (cf. the references above).

The aim of this article is to present a unified result for the case $p \in \big( \tfrac{2d}{d+1}, 2\big)$ where a smallness condition is posed at the normal velocity component $\vec{g_2} \dotn$, the tangential velocity component $\vec{g_2} \times \vec{n}$ and divergence data $g_1$ and which generalizes the previous approaches. 
Therefore, separate extensions for normal and tangential boundary data are constructed and employed in the existence proof. The regularity assumptions are kept minimal compared to the previous works. In order to focus on the key steps of the proof, we present this result in the setting of fluids with $(p,\delta)$-structure.

In the first section, notation is introduced and the assumptions and the main result, Theorem~\ref{thm:main}, are stated.
After that, we construct extensions of the boundary and divergence data.
Theorem~\ref{thm:main} is proved in Section~\ref{sec:proof}.
We then discuss the connection between Theorem~\ref{thm:main} and previous results as well as possible extensions to more general situations.
In the Appendix, we give more details about extensions of boundary and divergence data with certain regularity properties.

\section{Preliminaries, assumptions and main result}\label{sec:preliminaries}

\subsection{Notation}\label{sub:notation}
Vectors and tensors are written in bold letters while points and scalar quantities are not. Let $\R_\textit{sym}^{d\times d}$ denote the vector space of symmetric matrices of size $d$.
Throughout this article, we consider a bounded domain $\Omega \subset \R^d$ of dimension $d \in \{2, 3\}$ with at least Lipschitz continuous boundary. Hence, the normal vector $\vec{n}$ of unit length is defined almost everywhere on the boundary $\del \Omega$. 

We use standard notation for Lebesgue and Sobolev spaces and do not distinguish the notation of scalar-, vector- and matrix-valued function spaces.
The space $L_0^p(\Omega)$ contains precisely those $L^p$-functions
with mean value zero and we set $V_p := \{\vec{v} \in W_0^{1,p}(\Omega) \colon \div \vec{v} = 0\}$.
$\D{v}$ denotes the symmetric gradient of $\vec{v} \in W^{1,p}(\Omega)$. It gives rise to a norm of $V_p$ due to Korn's and Poincar{\'e}'s inequalities, which is the norm that we will usually employ on $V_p$.

A Banach space $X$ has a topological dual $X\d\!$ and $\fp{\cdot}{\cdot}_X$ denotes their canonical dual pairing. Indicating $X$ may be omitted if it is clear from the context. The conjugate exponent $p'$ is defined via $\frac{1}{p} + \frac{1}{p'} = 1$ for any $p \in [ 1, \infty ]$. We use the critical Sobolev exponent which is defined as $p\d := \frac{pd}{d-p} \in (p, \infty)$ for $p<d$ while $p\d := \infty$ is set otherwise.


\subsection{Assumptions}\label{sub:assumptions}
In this work, we consider extra stress tensors $\S$ which have
$(p,\delta)$-structure. This assumption generalizes the class of
power-law-models, but still assumes that the extra stress depends only
on the strain rate tensor $\D{v}$ and possesses a monotone, continuous strain-to-stress relation.
Albeit it seems possible to generalize the results to the extra stress tensors with variable exponents and, therefore, to the case of electrorheological fluids~(cf.~\cite{Sin}) or fluids with pressure- and shear-rate dependent viscosity~(cf.~\cite{lan09}), we stick to the simpler case for the ease of presentation.
As in~\cite{GNSE_DR07}, we assume:

\begin{assum}[$(p,\delta)$-structure]\label{assum:stress_tensor}
	The extra stress tensor $\S \colon \mathbb{R}_\textrm{sym}^{d
		\times d} \to \mathbb{R}_{\textrm{sym}}^{d \times d}$ has
	$(p,\delta)$-structure, i.e., it is continuous, satisfies $\S(\vec{0}) = \vec{0}$ and there exist constants $p \in (1, \infty)$, $\delta \geq 0$ and $C_1(\S)$, $C_2(\S) > 0$ such that
	\begin{equation*} \begin{split} 
	(\S(\vec{A}) - \S(\vec{B}))\cdot (\vec{A}-\vec{B}) &\geq C_1(\S)
	\, (\delta + \abs{\vec{B}} + \abs{\vec{A}-\vec{B}})^{p-2}
	\abs{\vec{A}-\vec{B}}^2,
	\\ 
	\abs{\S(\vec{A}) - \S(\vec{B})} &\leq C_2(\S) \, (\delta +
	\abs{\vec{B}} + \abs{\vec{A}-\vec{B}})^{p-2} \abs{\vec{A}-\vec{B}}
	\end{split}\end{equation*}
	holds for all $\vec{A}, \vec{B} \in \mathbb{R}_\text{sym}^{d \times d}$. The constants $C_1(\S), C_2(\S)$ and $p$ are called the characteristics of $\S$.
\end{assum}

We decompose the boundary data $\vec{g_2}$ into normal and tangential part, i.e.,
\begin{equation*}
\vec{g_n} := (\vec{g_2} \dotn) \vec{n}, \quad \vec{g_t} := \vec{g_2} - \vec{g_n}.
\end{equation*}

\begin{assum}[Growth and regularity]\label{assum:growth}
	We assume that $p \in (\tfrac{2d}{d+1}, 2)$ and define $s := \max \big\{ \big(\tfrac{p\d}{2} \big)', p \big\}$, $r :=\max \big\{  \tfrac{dp}{dp-2d+p},2p'\big\}$ and $q > r$.
	
	We assume $\vec{f} \in W_0^{1,p}(\Omega)\d$, $g_1 \in L^s(\Omega)$ and that $\del \Omega \in C^{k,1}$ with $k \in \{0, 1\}$. If $\del \Omega \in C^{1,1}$, we suppose $\vec{g_n} \in W^{1-\tfrac{1}{p},p}(\del\Omega) \cap W^{-\tfrac{1}{r},r}(\del\Omega)$ and $\vec{g_t} \in W^{1-\tfrac{1}{p},p}(\del\Omega) \cap W^{-\tfrac{1}{q},q}(\del\Omega)$. Otherwise, i.e.,~if $\del \Omega \in C^{0,1}$, we suppose $\vec{g_n} \in W^{1-\tfrac{1}{s},s}(\del\Omega)$ and $\vec{g_t} \in W^{1-\tfrac{1}{p},p}(\del\Omega) \cap L^q(\del \Omega)$.
	In both cases, we assume that the compatibility condition
	\begin{equation}\label{eq:s:compatibility}
	\int_{\Omega} g_1 \dx = \int_{\del \Omega} \vec{g_2} \dotn \dx
	\end{equation}
	is fulfilled.
	
	We abbreviate the norms
	\begin{gather*}
	K_d := \norm{g_1}_s,	\quad
	K_f := \norm{\vec{f}}_{W_0^{1,p}(\Omega)\d},\\
	K_n := \left\{ \begin{split}
	\norm{\vec{g_n}}_{1-\tfrac{1}{p},p} + \norm{\vec{g_n}}_{-\tfrac{1}{r},r} \qquad &\textrm{if } \del\Omega\in C^{1,1},\\
	\norm{\vec{g_n}}_{1-\tfrac{1}{s},s} \qquad &\textrm{else if } \del\Omega\in C^{0,1},
	\end{split}	\right.\\
	K_t := \left\{ \begin{split}
	\norm{\vec{g_t}}_{1-\tfrac{1}{p},p} + \norm{\vec{g_t}}_{-\tfrac{1}{q},q} \qquad &\textrm{if } \del\Omega\in C^{1,1},\\
	\norm{\vec{g_t}}_{1-\tfrac{1}{p},p} + \norm{\vec{g_t}}_{q} \qquad &\textrm{else if } \del\Omega\in C^{0,1}.
	\end{split}	\right.
	\end{gather*}
\end{assum}

\begin{rem}
	$p$ is the growth rate of the viscous stress tensor, $s$
	defines the regularity of the test functions and $r$ is chosen
	such that $\tfrac{1}{p} + \tfrac{1}{p\d} + \tfrac{1}{r} = 1$ is fulfilled at least. This is possible due to the restriction on $p$. Furthermore, the choice of $r$ and the restriction of $p$ imply that there is a continuous embedding $W^{1,s}(\Omega) \injto L^r(\Omega)$. $q > r$ is some extra integrability that will be used to create smallness in $L^r$.
\end{rem}


With this preparation, we may impose a smallness condition that will
prove to be sufficient for the existence of weak solutions.
\begin{cond}[Smallness of the data]\label{cond:data_small}
	The norms of $g_1$, $\vec{g_n}$, $\vec{g_t}$ defined in Assumption~\ref{assum:growth} are so small that there exists $\eta > 0$ s.t.
	\begin{align*}
	C_1&(\S) /  c(\Omega, p, q, \S)
	\\
	\geq	&\left[K_d + K_n\right]^{p-1} \big[ K_d^2 + K_n^2 + K_f
	+ \left[ K_t +  K_n +  K_d + \delta \right]^{p-1} \big]^{2-p}\\
	&+ \left[K_d + K_n\right]^{p-1} \big[\eta^{\tfrac{2}{r}-\tfrac{2}{q}} K_t^2\big]^{2-p}\\
	&+ \left[K_d + K_n\right]^{p-1} \big[\eta^{\tfrac{1}{p}-\tfrac{1}{q}-1} K_t\big]^{(p-1)(2-p)}\\
	&+ \big[\eta^{\tfrac{1}{r}-\tfrac{1}{q}} K_t\big]^{p-1} \big[ K_d^2 + K_n^2 + K_f
	+ \left[ K_t +  K_n +  K_d + \delta \right]^{p-1} \big]^{2-p}\\
	&+ \eta^{(\tfrac{1}{r}-\tfrac{1}{q})(3-p)} K_t^{3-p}\\
	&+ \eta^{\big(\tfrac{1}{r}-\tfrac{1}{q} + (\tfrac{1}{p}-\tfrac{1}{q}-1)(2-p)\big)(p-1)} K_t^{(p-1)(3-p)}\\
	=:& \;L(\eta)
	\end{align*}
	is fulfilled.
\end{cond}

\begin{rem}
	The constant $C_1(\S)$ in Condition \ref{cond:data_small} comes from the characteristics of $\S$ and $c(\Omega, p, q, \S)$ mainly consists of trace lifting, Sobolev embedding and Korn constants.
	I.e., the left-hand side does not depend on the data but can be considered as a ''fixed'' positive constant. The right-hand side depends monotonely on the data $K_d$, $K_t$, $K_n$, $K_f$ s.t. the condition is fulfilled if the data are sufficiently small.
	
	The (free) parameter $\eta$ controls the cutoff length for the tangential boundary data, cf.~Lemma~\ref{lem:extension_tang}. For general data, i.e., if $\vec{{h_\eta}}, \vec{k} \neq \vec{0}$,  the function $L(\eta)$ grows to infinity for both $\eta \to 0$ and $\eta \to \infty$: while the second, the forth and the fifth summand come with a positive exponent of $\eta$, $\eta$ has a negative exponent in the third summand. For the sixth summand, it depends on the values of $p$ and $q$ and the first one is obviously constant.
	Hence, $L$ has a minimum $\eta\d = \eta\d(K_d, K_t, K_n, K_f, \delta)$ which is the optimal choice for $\eta$. For consequences of this observation and for a discussion of some special cases, we refer to Section~\ref{sec:notes}.
\end{rem}

\subsection{Main result}\label{sub:result}

In this subsection, we give a weak formulation of~\eqref{eq:main_problem1},~\eqref{eq:main_problem2} and state our main result.

\begin{defi}[weak formulation]\label{defi:weak_form_pressure}
	Let $\Omega \subset \R^d$ be a bounded domain, $d\in \{2, 3\}$
	and let Assumptions~\ref{assum:stress_tensor} and~\ref{assum:growth} hold.
	A pair $(\vec{v}, \pi) \in W^{1,p}(\Omega) \times
	L^{s'}(\Omega)$ is called a weak solution
	of~\eqref{eq:main_problem1},~\eqref{eq:main_problem2} if it
	satisfies \eqref{eq:main_problem2} and 
	\begin{equation}\label{eq:weak_form_pressure}
	\fp{\S(\D{v})}{\D{\phi}}
	- \fp{\vec{v}\otimes\vec{v}}{\D{\phi}} 
	- \fp{\pi}{\div \vec{\phi}}
	= \fp{\vec{f}}{\vec{\phi}}
	\end{equation}
	for all $\vec{\phi} \in W_0^{1,s}(\Omega)$.
\end{defi}

\begin{rem}
	In some works, the convective term is considered in the form $\fp{\vec{v}\cdot \nabla\vec{v}}{\vec{\phi}}$ which, after formal integration by parts, yields a weak formulation involving the extra term $ - \fp{\vec{v}\div \vec{v}}{\vec{\phi}}$. It can be treated by the same methods and poses no additional difficulties.
\end{rem}

As a direct consequence of de Rham's theorem (cf.~\cite{Sohr}), we can rid ourselves of the pressure and obtain the following equivalent weak formulation:
\begin{lem}\label{lem:weak_form_sole}
	Let the assumptions of
	Definition~\ref{defi:weak_form_pressure} be satisfied and assume that $\vec{v} \in W^{1,p}(\Omega)$ satisfies~\eqref{eq:main_problem2} and
	\begin{equation*}
	\fp{\S(\D{v})}{\D{\phi}}
	- \fp{\vec{v}\otimes\vec{v}}{\D{\phi}} 
	= \fp{\vec{f}}{\vec{\phi}}
	\end{equation*}
	for all $\vec{\phi} \in V_s$. Then there is a pressure $\pi
	\in L^{s'}(\Omega)$ such that $(\vec{v}, \pi)$ is a weak
	solution of~\eqref{eq:main_problem1},~\eqref{eq:main_problem2}.
\end{lem}

The main result of this article reads as follows:

\begin{thm}\label{thm:main}
	Let $\Omega \subset \R^d$ be a bounded domain with boundary $\del \Omega \in C^{k,1}$, $k \in \{0, 1\}$, $d\in \{2, 3\}$ and $p \in \big( \tfrac{2d}{d+1}, 2 \big)$.
	We assume that $\S$, $\vec{f}$, $g_1$ and $\vec{g_2}$ fulfill Assumptions~\ref{assum:stress_tensor}, \ref{assum:growth} and Condition~\ref{cond:data_small}. Then there exists a weak solution of~\eqref{eq:main_problem1},~\eqref{eq:main_problem2}.
\end{thm}

\section{Extensions of boundary and divergence data}\label{sec:extension}

In the existence proof, we write the solution $\vec{v}$ as a sum $\vec{u} + \vec{g}$, where $\vec{u}$ is solenoidal and has zero boundary values. The subsequent estimates require the natural $W^{1,p}(\Omega)$ regularity of $\vec{g}$ as well as additional integrability $\vec{g} \in L^r(\Omega)$ with $r > p\d$ (at least for low values of $p$).
In order to get an efficient a priori estimate, establishing an extension $\vec{g}$ with small $L^r$-norm is crucial. To this end, we require $\vec{g} \in L^q(\Omega)$ for some $q >r$ for the tangential boundary data and then perform a cutoff argument. 


The main result of the section is as follows and will be proved at the end of the section:
\begin{prop} \label{prop:s:fs}
	Let $\Omega \subset \Rd$ be a bounded domain, let
	Assumption~\ref{assum:growth} be fulfilled and let $\eta > 0$.
	Then, there are extensions $\vec{h_\eta}, \vec{k} \in
	W^{1,p}(\Omega) \cap L^r(\Omega)$~such that 
	\begin{alignat*}{2}
	\div \vec{h_\eta} &= 0 && \textrm{in } \Omega\,,\\
	\tr \vec{h_\eta} &= \vec{g_t} && \textrm{on } \del\Omega\,,\\
	\norm{\vec{h_\eta}}_r 		&\leq c(\Omega, p, q) \eta^{\tfrac{1}{r}-\tfrac{1}{q}} K_t\,,\\
	\norm{\vec{h_\eta}}_{1,p} 	&\leq c(\Omega, p, q) \left[ 1 + \eta^{\tfrac{1}{p}-\tfrac{1}{q}-1} \right] K_t
	\end{alignat*}
	and
	\begin{alignat*}{2}
	\div \vec{k} &= g_1  && \textrm{in } \Omega\,, \\
	\tr \vec{k} &= \vec{g_n} && \textrm{on } \del\Omega\,, \\
	\max \{\norm{\vec{k}}_{1,p}, \norm{\vec{k}}_{r}\}
	&\leq c(\Omega, p) \left[ K_n + K_d \right]\,.
	\end{alignat*}
\end{prop}

\begin{rem}
	The estimates can be refined, e.g. s.t. the norm $\norm{\vec{h_\eta}}_r$ does not depend on the fractional derivative of $\vec{g_t}$ but only on the $W^{-\tfrac{1}{q},q}(\del\Omega)$/ the $L^q(\del\Omega)$-norm.
\end{rem}

Existence of appropriate extensions in the intersection space $W^{1,p}(\Omega) \cap L^q(\Omega)$ is guaranteed by the following two Propositions. The regularity requirements on the data differ slightly, depending on the regularity of the domain. If $\del \Omega \in C^{1,1}$, we get an optimal result via the theory of very weak solutions of the inhomogeneous Stokes equations:
\begin{prop} \label{prop:s:fs_reg}
	Let $p, t \in (1, \infty)$, $t = s\d$ and $\del \Omega \in C^{1,1}$.
	For $g_1 \in L^p(\Omega) \cap L^s(\Omega)$ and $\vec{g_2} \in W^{1-\tfrac{1}{p},p}(\del\Omega) \cap W^{-\tfrac{1}{t},t}(\del\Omega)$ that satisfy~\eqref{eq:s:compatibility}, there is an extension $\vec{g} \in W^{1,p}(\Omega) \cap L^t(\Omega)$ which satisfies
	\begin{align*}
	\div \vec{g} &= g_1 && \textrm{in } \Omega\,,\\
	\tr \vec{g} &= \vec{g_2} && \textrm{on } \del\Omega\,,\\
	\norm{\vec{g}}_{1,p} &\leq C (\norm{g_1}_p + \norm{\vec{g_2}}_{1-\tfrac{1}{p},p})\,,\\
	\norm{\vec{g}}_{t} &\leq C (\norm{g_1}_s + \norm{\vec{g_2}}_{-\tfrac{1}{t},t})\,.
	\end{align*}
\end{prop}

If the boundary $\del\Omega$ is only Lipschitz continuous, potential estimates on the Dirichlet Laplace problem yield a weaker assertion:
\begin{prop} \label{prop:s:fs_lipschitz}
	Let $p \in [\tfrac{2d}{d+1}, \tfrac{2d}{d-1}]$, $t \geq 2$, $\del \Omega \in C^{0,1}$ and $\vec{g_2} \in W^{1-\tfrac{1}{p},p}(\del\Omega) \cap L^t(\del\Omega)$.
	Then there is an extension $\vec{g} \in W^{1,p}(\Omega) \cap L^t(\Omega)$ s.t.
	\begin{align*}
	\tr \vec{g} &= \vec{g_2}\,,\\
	\norm{\vec{g}}_{1,p} &\leq C \norm{\vec{g_2}}_{1-\tfrac{1}{p},p}\,,\\
	\norm{\vec{g}}_{t} &\leq C \norm{\vec{g_2}}_{t}\,.
	\end{align*}
\end{prop}

The proofs of both Propositions \ref{prop:s:fs_reg} and \ref{prop:s:fs_lipschitz} are given in the Appendix. 

For the cutoff procedure, we present a well-known result on divergence
correction in a slightly more general way. Therefore, we define the
space
$$
H_{0,p} := \{ \vec{u} \in L^p(\Omega)\colon \div \vec{u} \in
L^p(\Omega), \vec{u} \dotn = 0 \textrm{ a.e. on } \Omega \}\,.
$$
We note that the space $C_0^\infty(\Omega)$ is dense in $H_{0,p}(\Omega)$ (cf.~\cite[Thm.~III.2.4]{Galdi}).
\begin{thm}[Divergence correction operator]\label{thm:bogovski2}
	Let $\Omega \subset \Rd$ be a bounded Lipschitz domain with $d \geq 2$ and $p,r \in (1,\infty)$. Then there exists a linear and bounded operator $\E \colon H_{0,p}(\Omega) \cap L^r(\Omega) \to W_0^{1,p}(\Omega) \cap L^r(\Omega)$ such that
	\begin{align*}
	\div \E \vec{f} &= \div \vec{f}\,, \\
	\norm{\E \vec{f}}_r &\leq c(\Omega, r) \norm{\vec{f}}_r\,,\\
	\norm{\E \vec{f}}_{1,p} &\leq c(\Omega, p) \norm{\div \vec{f}}_p
	\end{align*}
	holds for all $\vec{f} \in H_{0,p}(\Omega) \cap L^r(\Omega)$.
\end{thm}
\begin{proof}
	The proof is a consequence of the ''standard'' formulation of this theorem when $r=p$, cf.~e.g.~\cite[Thm. III.3.3]{Galdi}.
	According to~\cite{Galdi}, for any smooth $\vec{\tilde{f}} \in C_0^{\infty}(\Omega)$, we find a solution $\vec{v} \in C_0^{\infty}(\Omega)$ which satisfies
	\begin{equation*}
	\div \vec{v} = \div \vec{\tilde{f}}\,, \quad
	\norm{\vec{v}}_r \leq c(\Omega, r) \, \bignorm{\vec{\tilde{f}}}_r\,, \quad
	\norm{\vec{v}}_{1,p} \leq c(\Omega, p) \, \bignorm{\div \vec{\tilde{f}}}_p\,.
	\end{equation*}
	As its explicit construction reveals, $\vec{v}$ depends linearly on $\vec{\tilde{f}}$. Hence, we have a linear solution operator $\tilde{\E}\colon C_0^{\infty}(\Omega) \to C_0^{\infty}(\Omega)$.	
	The estimates on $\tilde{\E}$ yield its continuity. Thus, its extendability to the space $H_{0,p}(\Omega) \cap L^r(\Omega)$ follows by density. 
\end{proof}

\begin{lem}\label{lem:extension_tang}
	Let $\eta > 0$, $q > r \geq p$ and $\vec{h} \in W^{1,p}(\Omega) \cap L^q(\Omega)$ with $\vec{h} \dotn = 0$ a.e. on~$\partial\Omega$. Then there exists $\vec{h_\eta} \in W^{1,p}(\Omega) \cap L^r(\Omega)$ which satisfies
	\begin{equation}\begin{split}\label{eq:estim_eta}
	\div \vec{h_\eta} 			&= 0\,,\\
	\tr  \vec{h_\eta} 			&= \tr \vec{h}\,,\\
	\norm{\vec{h_\eta}}_r 		&\leq c(\Omega, p, q, r) \, \eta^{\tfrac{1}{r}-\tfrac{1}{q}} \norm{\vec{h}}_{q}\,,\\
	\norm{\nabla\vec{h_\eta}}_{p} 	&\leq c(\Omega, p, q, r) \left[ \norm{\nabla\vec{h}}_p + \eta^{\tfrac{1}{p}-\tfrac{1}{q}-1} \norm{\vec{h}}_{q} \right]\,.
	\end{split}\end{equation}
\end{lem}
\begin{proof}
	The core of this construction is due to~\cite{lan09}. For the convenience of the reader and in order to reduce some technical assumptions, we give a full proof here.
	
	
	Because of the triangle inequality, the distance function
	$d\colon \closed{\Omega} \to \R_{\geq 0}$ is Lipschitz
	continuous with constant $1$. Thus, we define the cutoff function
	\begin{equation*}
	\psi_\eta(x) := \left\{ \begin{array}{ll}
	1					& d(x) \leq \eta\,, \\
	2-\eta^{-1}d(x)	& \eta \leq d(x) \leq 2 \eta\,, \\
	0					& 2 \eta \leq d(x)\,.
	\end{array} \right.
	\end{equation*}
	Clearly, we have $\psi_\eta \in W^{1,\infty}(\Omega)$ and $\norm{\psi_\eta}_\infty = 1$. Furthermore, we have $\norm{\psi_\eta}_{\ell} \leq \abs{\supp \psi_\eta}^{\tfrac{1}{\ell}} \leq (c(\Omega)\eta)^{\tfrac{1}{\ell}}$ and $\norm{\nabla\psi_\eta}_{\ell} \leq (c(\Omega)\eta)^{\tfrac{1}{\ell}-1}$ for any ${\ell} \in (1, \infty)$.

	The auxiliary function $\vec{\tilde{h_\eta}} := \psi_\eta \vec{h} \in W^{1,p}(\Omega) \cap L^r(\Omega)$ fulfills
	\begin{equation}\label{eq:h_tilde_r}
	\bignorm{\vec{\tilde{h_\eta}}}_r
	\leq \norm{\psi_\eta}_{\ell} \norm{\vec{h}}_q
	\leq (c(\Omega) \eta )^{\tfrac{1}{r}-\tfrac{1}{q}} \norm{\vec{h}}_{q}
	\end{equation}
	with $\tfrac{1}{\ell} := \tfrac{1}{r} - \tfrac{1}{q}$ due to the Hölder inequality. Similarly, we obtain with $\tfrac{1}{m} := \tfrac{1}{p} - \tfrac{1}{q}$ that
	\begin{align}
	\begin{aligned}
	\bignorm{\nabla\vec{\tilde{h_\eta}}}_p
	&\leq \norm{\nabla \psi_\eta}_{m} \norm{\vec{h}}_q + \norm{\psi_\eta}_\infty \norm{\nabla\vec{h}}_p  \\
	&\leq c(\Omega, p,q,r) \Big[\norm{\nabla\vec{h}}_p+ \eta
	^{\tfrac{1}{p}-\tfrac{1}{q}-1} \norm{\vec{h}}_{q} \Big]\,.
	\end{aligned}
	\label{eq:h_tilde_1p}
	\end{align}
	
	So, $\vec{\tilde{h_\eta}}$ has the correct boundary data and fulfills the correct estimates, but is not solenoidal. For a divergence correction, we use the operator $\E$ from Theorem~\ref{thm:bogovski2} where we employ that $\vec{\tilde{h_\eta}}$ has vanishing normal trace.
	So, we set $\vec{{h_\eta}} := \vec{\tilde{h_\eta}} - \E \vec{\tilde{h_\eta}} \in W^{1,p}(\Omega) \cap L^r(\Omega)$. The claimed estimates follow by~\eqref{eq:h_tilde_r}, \eqref{eq:h_tilde_1p} and the estimates for $\E$.
\end{proof}

\begin{proof}[Proof of Proposition \ref{prop:s:fs}]
	The construction of $\vec{h_\eta}$ and $\vec{k}$ is illustrated in Figure~\ref{fig:s:extensions_workflow}.
	
	If the boundary is sufficiently regular, i.e., $\del \Omega \in C^{1,1}$, we find extensions ${\vec{k} \in W^{1,p}(\Omega) \cap L^r(\Omega)}$ of $(g_1, \vec{g_n})$ and $\vec{h} \in W^{1,p}(\Omega) \cap L^q(\Omega)$ of $(0, \vec{g_t})$ via Proposition~\ref{prop:s:fs_reg}. For $\vec{h}$, a cutoff procedure according to Lemma~\ref{lem:extension_tang} then yields $\vec{h_\eta} \in W^{1,p}(\Omega) \cap L^r(\Omega)$.
	
	If the boundary is only Lipschitz continuous, we apply the
	trace lifting operator to extend $\vec{g_n}$ to
	$\vec{\tilde{k}} \in W^{1,s}(\Omega)$. Then we use the
	classical Bogovski\u{\i} operator to get $\vec{k} \in
	W^{1,s}(\Omega)$ such that $\div \vec{k} = g_1$, $\tr \vec{k} = \vec{g_2}$. Since $s\d \geq r$, a Sobolev embedding concludes the estimate on $\vec{k}$ (cf.~\cite{Galdi}, \cite[Lemma 2.3]{JR21_inhom} for more details).
	For the tangential boundary data $\vec{g_t}$, we apply Proposition \ref{prop:s:fs_lipschitz} to get an extension $\vec{h} \in W^{1,p}(\Omega) \cap L^q(\Omega)$. Again, a cutoff argument according to Lemma~\ref{lem:extension_tang} yields $\vec{h_\eta} \in W^{1,p}(\Omega) \cap L^r(\Omega)$.
\end{proof}

\begin{figure}[t]
	\centering
	\begin{tikzpicture}[font=\small,thick]
	
	\node[draw,
	diamond,
	minimum width=2.5cm,
	inner sep=0] (block0) { $\del\Omega \in C^{1,1}$};
	
	\node[draw,
	rounded rectangle,
	below left= 5mm and 15mm of block0,
	minimum width=2.5cm,
	minimum height=1cm,
	align=center] (block1a) { Suppose data $g_1 \in L^s(\Omega)$,\\ $\vec{g_n} \in W^{1-\tfrac{1}{p},p}(\del\Omega) \cap W^{-\tfrac{1}{r},r}(\del\Omega)$\\ $\vec{g_t} \in W^{1-\tfrac{1}{p},p}(\del\Omega) \cap W^{-\tfrac{1}{q},q}(\del\Omega)$ };
	
	\node[draw,
	rounded rectangle,
	below right= 5mm and 15mm of block0,
	minimum width=2.5cm,
	minimum height=1cm,
	align=center] (block1b) { Suppose data $g_1 \in L^s(\Omega)$,\\ $\vec{g_n} \in W^{1-\tfrac{1}{s},s}(\del\Omega)$, \\ $\vec{g_t} \in W^{1-\tfrac{1}{p},p}(\del\Omega) \cap L^q(\del\Omega)$ };
	
	\node[draw,
	below =5cm of block1a,
	minimum width=3.5cm,
	minimum height=1cm,
	align=center
	] (block2a) { Extension \\ $\vec{k} \in W^{1,p}(\Omega) \cap L^r(\Omega) $ \\ of $g_1$, $\vec{g_n}$ };
	
	\node[draw,
	below =of block1b,
	minimum width=3.5cm,
	minimum height=1cm,
	align=center
	] (block2c) { Extension \\ $\tilde{\vec{k}} \in W^{1,s}(\Omega) $ of $\vec{g_n}$ };
	
	\node[draw,
	below left = 37mm and -25mm of block0,
	minimum width=3.5cm,
	minimum height=1cm,
	align=center
	] (block2b) { Extension \\ $\vec{h} \in W^{1,p}(\Omega) \cap L^q(\Omega) $ of $\vec{g_t}$ };

	\node[draw,
	below=of block2c,
	minimum width=3.5cm,
	minimum height=1cm,
	align=center
	] (block3c) { Extension $\vec{k} \in W^{1,s}(\Omega) $ \\ of $g_1$, $\vec{g_n}$ };
	
	\node[draw,
	below= 25mm of block2b,
	minimum width=3.5cm,
	minimum height=1cm,
	align=center
	] (block3b) { Solenoidal extension \\ $\vec{h_\eta} \in W^{1,p}(\Omega) \cap L^r(\Omega) $ \\ with small $L^r$-norm. };
	
	\node[draw,
	below= of block3c,
	minimum width=3.5cm,
	minimum height=1cm,
	align=center
	] (block4c) { Extension \\ $\vec{k} \in W^{1,p}(\Omega) \cap L^r(\Omega) $ \\ of $g_1$, $\vec{g_n}$ };
	
	
	\draw[-latex] (block0) -| (block1a)
	node[pos=0.25,fill=white,inner sep=0]{Yes};
	
	\draw[-latex] (block0) -| (block1b)
	node[pos=0.25,fill=white,inner sep=0]{No};
	
	\draw[-latex] (block1b) edge node[pos=0.4,fill=white,inner sep=2pt]{Trace lift}(block2c);
	
	\draw[-latex] (block1a) edge node[pos=0.4,fill=white,inner sep=2pt,align=center]{Stokes problem \\ Prop. \ref{prop:s:fs_reg}}(block2a);
	
	\draw[-latex] (block1a) edge node[pos=0.4,fill=white,inner sep=2pt,align=center]{Stokes problem \\ Prop. \ref{prop:s:fs_reg}}(block2b);
	
	\draw[-latex] (block1b) edge node[pos=0.4,fill=white,inner sep=2pt,align=center]{Dirichlet problem \\ Prop. \ref{prop:s:fs_lipschitz}}(block2b);
	
	\draw[-latex] (block2c) edge node[pos=0.4,fill=white,inner sep=2pt]{Divergence correction}(block3c);
	
	\draw[-latex] (block2b) edge node[pos=0.4,fill=white,inner sep=2pt,align=center]{Cutoff + divergence \\ correction, Lemma \ref{lem:extension_tang}}(block3b);
	
	\draw[-latex] (block3c) edge node[pos=0.4,fill=white,inner sep=2pt]{Sobolev embedding}(block4c);
	
	\end{tikzpicture}
	\caption{Construction of the extensions of boundary and divergence data for Proposition \ref{prop:s:fs}.}
	\label{fig:s:extensions_workflow} 
\end{figure}
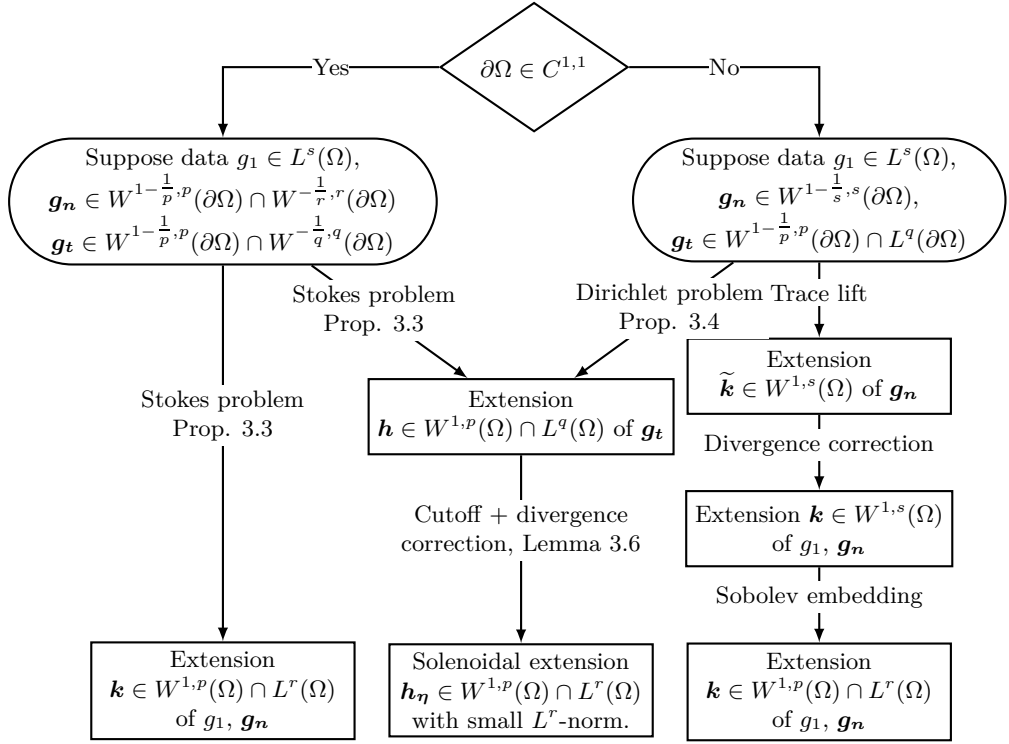

\section{Proof of the main result}\label{sec:proof}

In this section, we prove Theorem~\ref{thm:main} and therefore suppose
that the definitions and assumptions of that Theorem hold. In particular, we take the decompositions $\vec{g_2} = \vec{g_t} + \vec{g_n}$ and $\vec{g} = \vec{h_\eta} + \vec{k}$ based on the extensions from Proposition~\ref{prop:s:fs} as given.

\subsection{A priori estimate}\label{sub:estimates}

The viscous stress tensor possesses the following properties:
\begin{lem}[viscous stress tensor]\label{lem:properties_s}
	The extra stress tensor $\S$ induces a bounded and continuous operator $\vec{S}\colon W_0^{1,p}(\Omega) \to W_0^{1,p}(\Omega)\d$ via
	\begin{equation} \label{eq:defi_S}
	\fp{\vec{S}(\vec{u})}{\vec{\phi}} := \fp{\S(\vec{Du}+\vec{Dg})}{\vec{D\phi}}.
	\end{equation}
	It admits the lower bound
	\begin{equation}\label{eq:estim_s}
	\fp{\vec{S}(\vec{u})}{\vec{u}} 
	\geq G_1 \norm{\vec{Du}}_p^p - G_{31} \norm{\vec{Du}}_p 
	\end{equation}
	for all $\vec{u} \in W_0^{1,p}(\Omega)$ with constants
	\begin{equation*}
	G_1 := c(p) \, C_1(\S) \textrm{ and } G_{31} := c(\Omega, p, q, \S) \Big[ \eta^{\tfrac{1}{p}-\tfrac{1}{q}-1} K_t + K_t +  K_n +  K_d + \delta \Big]^{p-1}.
	\end{equation*}
\end{lem}

\begin{proof}
	The proof follows standard arguments which are carried out in detail e.g. in~\cite{Cetraro, JR21_inhom}.
	In \cite[Lemma 2.13]{JR21_inhom}, it is shown that Assumption
	\ref{assum:stress_tensor} implies 	for $\vec{v} \in
	W_0^{1,p}(\Omega)$ the lower bound 
	\begin{align}\label{eqh:estim_viscous}
	\begin{aligned}
	\fp{\vec{S}(\vec{u})}{\vec{u}} &\geq c(p) \, C_1(\S)
	\norm{\vec{Du}}_p^p
	\\
	&\quad - c(p) \, \big ( C_1(\S) + C_2(\S)
	\big) \bignorm{\abs{\vec{Dg}}+\delta}_p^{p-1}
	\norm{\vec{Du}}_p\,.
	\end{aligned}
	\end{align}
	We further estimate
	\begin{align*}
	\bignorm{\abs{\vec{Dg}}+\delta}_p
	&\leq \norm{\D{h_\eta}}_p + \norm{\D{k}}_p
	+ \delta \abs{\Omega}^{\tfrac{1}{p}}\\
	&\leq c(\Omega, p, q) \Big[ \eta^{\tfrac{1}{p}-\tfrac{1}{q}-1} K_t
	+ K_t +  K_n +  K_d + \delta \Big]\, .
	\end{align*}
	Inserting this into~\eqref{eqh:estim_viscous} yields the lower bound.
\end{proof}

In order to exploit the smallness of the tangential extension
$\vec{h_\eta}$ in the best possible way, we estimate the convective
term such that only the $L^r$-norm and not the gradient norm of $\vec{g}$ is involved. This step is responsible for the lower bound $p > \tfrac{2d}{d+1}$.

\begin{lem}[convective term]\label{lem:properties_t}
	The convective term induces a bounded and continuous operator $\vec{T}\colon V_p \to W_0^{1,s}(\Omega)\d$ defined via
	\begin{equation} \label{eq:defi_T}
	\fp{\vec{T}(\vec{u})}{\vec{\phi}}
	:= - \fp{(\vec{u}+\vec{g}) \otimes (\vec{u}+\vec{g})}{\vec{D\phi}}\,.
	\end{equation}
	For $\sigma \geq s$ with $\sigma > \big (\frac{p\d}{2}\big )'$, $\vec{T}$ is a strongly continuous operator from $V_p$ to $W_0^{1,\sigma}(\Omega)$. If $\vec{u} \in V_\sigma$, we have
	\begin{equation}\label{eq:estim_t}
	\abs{\fp{\vec{T}(\vec{u})}{\vec{u}}}
	\leq G_2 \norm{\D{u}}_p^2 + G_{32} \norm{\D{u}}_p
	\end{equation}
	with constants
	\begin{equation*}
	G_2 := c(\Omega, p, q) \Big[K_d + K_n + \eta^{\tfrac{1}{r}-\tfrac{1}{q}} K_t \Big] \textrm{ and } G_{32} := c(\Omega, p, q) \Big[K_d^2 + K_n^2 + \eta^{\tfrac{2}{r}-\tfrac{2}{q}} K_t^2\Big]\,.
	\end{equation*}
\end{lem}
\begin{proof}
	The definition of $s$ implies $\tfrac{1}{p\d} + \tfrac{1}{p\d} + \tfrac{1}{s} \leq 1$, so both well-definedness and boundedness of $\vec{T}(\vec{u})$ follow from Hölder's inequality. The continuity and strong continuity follow from Hölder's inequality in conjuction with the embeddings $W^{1,p}(\Omega) \injto L^{p\d}(\Omega)$, $W^{1,p}(\Omega) \injto\injto L^{\tilde{p}}(\Omega)$ for $\tilde{p} < p\d$, respectively (cf.~\cite[Lemma 2.19]{JR21_inhom} for details).
	
	Let $\vec{u} \in V_\sigma$. We reformulate the convective term as
	\begin{equation*}
	\fp{\vec{T}(\vec{u})}{\vec{u}}
	= 
	- \fp{\vec{u} \cdot \nabla\vec{u}}{\vec{g}}
	+ \tfrac{1}{2} \fp{g_1 \vec{u}}{\vec{u}}
	- \fp{\vec{g} \cdot \nabla\vec{u}}{\vec{g}}\,.
	\end{equation*}
	With the Hölder, Sobolev and Korn inequalities, we estimate
	\begin{align*}
	\abs{\fp{\vec{T}(\vec{u})}{\vec{u}}}
	&\leq \norm{\vec{g}}_r \norm{\nabla\vec{u}}_p \norm{\vec{u}}_{p\d}
	+ \tfrac{1}{2} \norm{g_1}_s \norm{\vec{u}}_{p\d}^2
	+ \norm{\vec{g}}_{2p'}^2 \norm{\nabla\vec{u}}_p\\
	&\leq c(\Omega, p) \big[ (\norm{\vec{g}}_r + \norm{g_1}_s) \norm{\D{u}}_p^2
	+ \norm{\vec{g}}_r^2 \norm{\D{u}}_p \big].
	\end{align*}
	
	We apply the decomposition $\vec{g} = \vec{k} + \vec{h_\eta}$,
	Minkowski's and Young's inequalities and insert the estimates
	from Proposition \ref{prop:s:fs} to obtain
	\begin{align*}
	&\abs{\fp{\vec{T}(\vec{u})}{\vec{u}}} \\
	&\leq c(\Omega, p) \Big[ (\norm{\vec{k}}_r + \norm{g_1}_s + \norm{\vec{h_\eta}}_r)  \norm{\D{u}}_p^2
	+  (\norm{\vec{k}}_r^2 + \norm{\vec{h_\eta}}_r^2) \norm{\D{u}}_p \Big]  \\
	&\leq c(\Omega, p, q) \Big[ (K_d + K_n + \eta^{\tfrac{1}{r}-\tfrac{1}{q}} K_t) \norm{\D{u}}_p^2
	+ (K_d^2 + K_n^2 + \eta^{\tfrac{2}{r}-\tfrac{2}{q}} K_t^2)  \norm{\D{u}}_p \Big] \,.
	\end{align*}
\end{proof}

As an auxiliary tool for simplifying terms in the next step, we use a reverse Jensen-type inequality:
\begin{lem}\label{lem:reverse_jensen}
	Let $\alpha \in (0,1)$ and $x_1, ..., x_n$ be non-negative real numbers. Then it holds
	\begin{equation}
	\left[\sum_{i=1}^{n} x_i \right]^\alpha \leq \sum_{i=1}^{n} x_i^\alpha.
	\end{equation}
\end{lem}
\begin{proof}
	For $n < 2$, the assertion is trivial. For $n=2$, we may w.l.o.g. consider the case $x_2=1$. Since the estimate is true for $x_1=0$, $f(x_1) = x_1^\alpha + 1 - (x_1 + 1)^\alpha$ is continuous for $x_1 \geq 0$ and its derivative $f'(x_1) = \alpha \left( x_1^{\alpha-1} - (x_1 + 1)^{\alpha-1} \right)$ is non-negative for $x_1 > 0$, the estimate is also proved in this case. The general statement follows then by induction.
\end{proof}

With these preparations, we are ready to prove an a priori estimate:

\begin{lem}[a priori estimate]\label{lem:lower_estimate}
	Let $\sigma$ as in Lemma~\ref{lem:properties_t}. For $\vec{u} \in V_\sigma$, it holds
	\begin{equation}\label{eq:apriori1}
	\fp{\vec{S}(\vec{u}) - \vec{T}(\vec{u}) - \vec{f}}{\vec{u}}
	\geq G_1 \norm{\vec{Du}}_p^p - G_2 \norm{\vec{Du}}_p^2 - G_3 \norm{\vec{Du}}_p
	\end{equation}
	with the constants $G_1$ and $G_2$ defined in Lemmas~\ref{lem:properties_s} and~\ref{lem:properties_t} and $G_3 := G_{31} + G_{32} + c(\Omega, p) K_f$.
	
	If, in addition, Condition~\ref{cond:data_small} is fulfilled, then we have
	\begin{equation}\label{eq:apriori}
	\fp{\vec{S}(\vec{u}) - \vec{T}(\vec{u}) - \vec{f}}{\vec{u}}
	\geq 0
	\end{equation}
	on the sphere $\norm{\D{u}}_p = \left[\tfrac{2G_3}{G_1}\right] ^{\tfrac{1}{p-1}} =: R$.
\end{lem}

\begin{proof}
	The first estimate~\eqref{eq:apriori1} follows directly
	from~\eqref{eq:estim_s},~\eqref{eq:estim_t}, H\"older's
	inequality for the term with $\vec{f}$ and the definition of $G_3$.
	
	For the second estimate~\eqref{eq:apriori}, if $\norm{\D{u}}_p = R$, we observe that it is sufficient to guarantee
	\begin{equation} \label{eq:condition_g}
	G_1 \geq 2 G_2^{p-1} G_3^{2-p}
	\end{equation}
	since this implies
	\begin{align*}
	\fp{\vec{S}(\vec{u}) - \vec{T}(\vec{u}) - \vec{f}}{\vec{u}}
	&\geq G_1 R^p - G_2 R^2 - G_3 R\\
	&\geq R^p \left[\left( \tfrac{G_1}{2} - G_2 R^{2-p}\right) + \left(\tfrac{G_1}{2} - G_3 R^{1-p}\right) \right]\\
	&\geq R^p \left[ 0 + 0 \right] = 0\,,
	\end{align*}
	using the definition of $R$.
	
	We notice that the left-hand side of~\eqref{eq:condition_g}, $G_1$, does not depend on the data but merely on the viscous stress tensor which we assume to be fixed. We insert the definition of $G_2$, $G_3$ and Lemma~\ref{lem:reverse_jensen} to estimate
	\begin{equation*}
	G_2^{p-1} \leq c(\Omega, p, q) \Big[
	\left[K_d + K_n\right]^{p-1} + \big[\eta^{\tfrac{1}{r}-\tfrac{1}{q}} K_t\big]^{p-1} \Big]
	\end{equation*}
	and
	\begin{align*}
	G_3^{2-p} \leq 
	c(\Omega, p, q, \S) &\Big[
	\big[ K_d^2 + K_n^2 + K_f
	+ \left[ K_t +  K_n +  K_d + \delta \right]^{p-1} \big]^{2-p}\\
	&+ \big[\eta^{\tfrac{2}{r}-\tfrac{2}{q}} K_t^2\big]^{2-p}
	+ \big[\eta^{\tfrac{1}{p}-\tfrac{1}{q}-1} K_t\big]^{(p-1)(2-p)}
	\Big]\,.
	\end{align*}
	
	Together with~\eqref{eq:condition_g}, this yields
	\begin{align*}
	&2 G_2^{p-1} G_3^{2-p}\\
	&\leq \, c(\Omega,p,q,\S) \Bigg[ \left[K_d + K_n\right]^{p-1} \big[ K_d^2 + K_n^2 + K_f
	+ \left[ K_t +  K_n +  K_d + \delta \right]^{p-1} \big]^{2-p}\\
	& \quad + \left[K_d + K_n\right]^{p-1} \big[\eta^{\tfrac{2}{r}-\tfrac{2}{q}} K_t^2\big]^{2-p}
	+ \left[K_d + K_n\right]^{p-1} \big[\eta^{\tfrac{1}{p}-\tfrac{1}{q}-1} K_t\big]^{(p-1)(2-p)}\\
	& \quad + \big[\eta^{\tfrac{1}{r}-\tfrac{1}{q}} K_t\big]^{p-1} \big[ K_d^2 + K_n^2 + K_f
	+ \left[ K_t +  K_n +  K_d + \delta \right]^{p-1} \big]^{2-p}\\
	& \quad + \eta^{(\tfrac{1}{r}-\tfrac{1}{q})(3-p)} K_t^{3-p}
	+ \eta^{\big(\tfrac{1}{r}-\tfrac{1}{q} + (\tfrac{1}{p}-\tfrac{1}{q}-1)(2-p)\big)(p-1)} K_t^{(p-1)(3-p)}\Bigg]\\
	&= \, c(\Omega,p,q,\S) L(\eta) \leq G_1\,,
	\end{align*}
	where we used Condition~\ref{cond:data_small} in the last step.
\end{proof}

\subsection{Existence proof for $p \in \big( \tfrac{3d}{d+2}, 2 \big)$}\label{sub:proof_medium}

In this case, we have $s=p$. Hence, the test function space $V_s$ agrees with the solution function space $V_p$ and we may directly apply a local coercivity generalization of the main theorem on pseudomonotone operators due to Br{\'e}zis, Hess and Kato~\cite{Zeidler2B}.

Indeed, $V_p$ is reflexive and separable as a closed subspace of $W_0^{1,p}(\Omega)$.
We abbreviate $\vec{P}\colon V_p \to V_p\d: \vec{u} \mapsto \vec{P}\vec{u} := \vec{S}\vec{u} - \vec{T}\vec{u} - \vec{f}$ and search for a zero of $\vec{P}$.
Due to Lemmas~\ref{lem:properties_s} and~\ref{lem:properties_t}, $\vec{P}$ is bounded, continuous and pseudomonotone.
Moreover, according to Lemma~\ref{lem:lower_estimate}, $\vec{P}$ is locally coercive with radius~${R= \big[\tfrac{2G_3}{G_1}\big] ^{\tfrac{1}{p-1}}}$. Thus, existence follows from the main theorem on pseudomonotone operators.

\subsection{Existence proof for $p \in \big( \tfrac{2d}{d+1}, \tfrac{3d}{d+2} \big]$}\label{sub:proof_low}

In this case, we have $s = \big (\frac{p\d}{2}\big )'$. As above, we consider the operator $\vec{P}\colon V_p \to V_s\d$, $\vec{u} \mapsto \vec{P}\vec{u} := \vec{S}\vec{u} - \vec{T}\vec{u} - \vec{f}$ and search for a zero of $\vec{P}$.
We define a sequence of approximate problems which can be solved with pseudomonotone operator theory and then pass to the limit with the Lipschitz truncation method. This part of the proof is similar to~\cite{JR21_inhom}.

We choose an exponent $\sigma > \max \{s, 2\}$ and define the symmetric $\sigma$-Laplacian $\vec{A} \colon V_\sigma \to V_\sigma\d$ as
\begin{equation*}
\fp{\vec{A}(\vec{u})}{\vec{\phi}} := \fp{\abs{ \vec{Du}}^{\sigma-2} \vec{Du}}{\vec{D\phi}}\,. 
\end{equation*}
As a well-known fact, $\vec{A}$ is bounded, continuous and monotone.
For sufficiently large $n \in \N$, we consider the regularized problem
\begin{equation} \label{eq:tq}
\fp{\vec{P}(\vec{u^n})}{\vec{\phi}} + \tfrac{1}{n} \fp{\vec{A}(\vec{u^n})}{\vec{\phi}} = 0 
\end{equation}
for $\vec{\phi} \in V_\sigma$. 
We prove existence of an approximate solution in the space $V_\sigma^n$, which is $V_\sigma$ with the equivalent norm 
$\lVert {\vec{u}} \rVert _{\sigma, n} := \max \big\{ n^\frac{-2}{2\sigma-1} \norm{\vec{Du}}_\sigma, \norm{\vec{Du}}_p \big\}$.
Clearly, the corresponding operator $\vec{P} + \tfrac{1}{n}\vec{A}$ is bounded, continuous and pseudomonotone.

We show local coercivity with radius $R= \big[\tfrac{2G_3}{G_1}\big] ^{\tfrac{1}{p-1}}$. If $\norm{\D{u}}_p = R$, then this is a direct consequence of Lemma~\ref{lem:lower_estimate}. Else if $\norm{\D{u}}_\sigma = n^\frac{2}{2\sigma-1} R$, we use the continuous embedding~$V_\sigma \to V_p$ to estimate
\begin{equation*}
\fp{\vec{P}(\vec{u^n})}{\vec{u^n}} + \tfrac{1}{n} \fp{\vec{A}(\vec{u^n})}{\vec{u^n}}
> -c \, G_2 R^2 -c \, G_3 R + n^\frac{1}{2\sigma-1} R^\sigma\,.
\end{equation*}
As $n^\frac{1}{2\sigma-1}$ grows to infinity for $n \to \infty$, the right hand side is positive for sufficiently large $n$. By the main theorem on pseudomonotone operators, we find approximate solutions $\vec{u^n} \in V_{\sigma,n}$ which come with the a priori estimate
\begin{equation}\label{eq:apriori_proof}
\max \left\{ n^\frac{-2}{2\sigma-1} \norm{\vec{Du}}_\sigma, \norm{\vec{Du}}_p \right\} \leq R\,.
\end{equation}
Due to weak compactness, we find a weakly convergent (and renamed)
subsequence $\vec{u^n} \weakto \vec{u}$ in $V_p$. Moreover,
\eqref{eq:apriori_proof} yields $\norm{n^{-1}
	\vec{A}(\vec{u^n})}_{\sigma'} \leq n^\frac{1}{2\sigma-1}
R^{\sigma-1} \to 0$ as $n \to \infty$. The convergence to the limit
equation $\vec{P}(\vec{u}) = \vec{0}$ then follows with the help of
the Lipschitz truncation method as a consequence
of~\cite[Thm.~2.32]{JR21_inhom}.

\section{Discussion and comparison to previous results}\label{sec:notes}

The proof in Section~\ref{sec:proof} shows that the critical point for
existence of weak solutions are conditions on the data such that the a priori estimate~\eqref{eq:apriori} holds.
Condition~\ref{cond:data_small} quantifies that normal and tangential data have to be small as a product:
The first summand in the definition of $L(\eta)$ contains the product of $K_d+K_n$ and $K_t$ and is independent of $\eta$.

For $p > \tfrac{2d}{d+1}$, the results of \cite{mikelic, Sin, JR21_inhom} as well as \cite[Thm.~1]{Elshehabey21} follow qualitatively from Theorem~\ref{thm:main}: there, we have $\vec{g_n} = \vec{g_2}$ and $\vec{g_t} = \vec{0}$. We point out that the above proof does not use that $\vec{g_n}$ is a normal vector field, so both tangential and normal boundary data may be put into $\vec{g_n}$.
In consequence, all $\eta$-dependent terms in the definition of $L$ vanish and only the first summand remains, which yields their smallness condition.

If the normal and divergence data are correlated to the tangential data as $\max \{\norm{\vec{g_n}}, \norm{g_1} \} \leq c \eta^{\tfrac{1}{r}-\tfrac{1}{q}} \norm{\vec{g_t}}$, we obtain $K_n \sim K_d \sim \eta^{\tfrac{1}{r}-\tfrac{1}{q}}$.
In consequence, all summands of $L(\eta)$ come with positive powers of $\eta$ if $p > 2- \tfrac{1}{d}$ and if $q$ is sufficiently large. Condition~\ref{cond:data_small} is then automatically fulfilled for sufficiently small $\eta$. This was first noticed for homogeneous divergence data in~\cite[Thm.~2]{Elshehabey21}.
We note that letting $\eta$ tend to zero also forces the
normal and divergence data to be small. Additionally, larger
tangential data will result in smaller values
of~$\eta$. Therefore, we interpret this result such that for any admissible tangential data $\vec{g_t}$, there is a perturbation radius within which non-vanishing normal and divergence data are allowed.

In the case of purely tangential boundary data, our argumentation boils down to the following:
\begin{cor}\label{cor:repro_lanz}
	If $\vec{g_n} = \vec{0}$, $g_1 = 0$, $p > 2-\tfrac{1}{d}$ and $q > \tfrac{d(3-p)}{dp-2d+1}$, then Condition~\ref{cond:data_small} is fulfilled for arbitrary data $\vec{g_t}$. 
\end{cor}
\begin{proof}
	In that case, $K_n = K_d = 0$ and the first three terms in the
	definition of $L$ vanish. The forth and the fifth term tend to
	zero for $\eta \to 0$. The same holds for the sixth term,
	since the condition on $p$ implies that we can choose $q >
	\tfrac{d(3-p)}{dp-2d+1}$ such that $\eta$ comes with a positive exponent.
	Hence, $L(\eta) \to 0$ as $\eta \to 0$ and Condition~\ref{cond:data_small} comes true for sufficiently small $\eta$.
\end{proof}
\begin{rem}
	Corollary~\ref{cor:repro_lanz} implies the results in \cite{lan09,lan08} in the case of pressure-independent stress tensors.
\end{rem}

Finally, we discuss a few questions on the extendability of our results.
In~\cite{mikelic,JR21_inhom,Elshehabey21}, existence of weak solutions has been shown for the wider range ${p > \tfrac{2d}{d+2}}$. Therefore, the estimate $\abs{\fp{\vec{u}\otimes\vec{u}}{\nabla\vec{g}}} \leq \norm{\vec{u}}_{p\d}^2 \norm{\nabla \vec{g}}_s$ for (part of) the convective term is central. This is still possible in our case (under suitable regularity assumptions), but would bring no advantage since we could not make use of the smallness~\eqref{eq:estim_eta}$_3$ anymore but instead had to work with the gradient norm which scales critically as $\eta \to 0$.
Hence, our approach is useful only in the range where we can estimate $\abs{\fp{\vec{u} \cdot \nabla\vec{u}}{\vec{g}}} \leq \norm{\vec{u}}_{p\d} \norm{\nabla\vec{u}}_{p} \norm{\vec{g}}_r$ which yields the lower bound on~$p$.

The analogue inhomogeneous system for electrorheological fluids naturally leads to a functional setting within Sobolev spaces with variable exponents~$W^{1,p(\cdot)}(\Omega)$. Similar to~\cite{Sin}, Theorem~\ref{thm:main} can be generalized in a straight-forward way to this setting under the condition $p_{-} = \inf_{x \in \Omega} p(x) > \tfrac{2d}{d+1}$.

\appendix
\section{Appendix}

\subsection{Solution for regular domains}

Throughout this subsection, we let $\del\Omega \in C^{1,1}$.
We construct the extensions as the solutions of a homogeneous (in force) Stokes system.

The weak solvability of the steady Stokes system is guaranteed by the following Proposition~\cite[Corr.~4.15--~Thm.~4.18]{Decomposition_Amrouche98}:
\begin{prop}
	Let $p \in (1, \infty)$, $\vec{f} \in W^{-1,p}(\Omega)$, $g_1 \in L^p(\Omega)$ and $\vec{g_2} \in W^{1-\tfrac{1}{p},p}(\del\Omega)$ satisfy~\eqref{eq:s:compatibility}.
	Then, there is a unique weak solution $(\vec{u}, \pi) \in W^{1,p}(\Omega) \times L_0^p(\Omega)$ to
	\begin{equation} \label{eq:stokes_inhom} \begin{split}
	- \Delta \vec{u} + \nabla \pi &= \vec{f} \quad \textrm{in } \Omega\,,\\
	\div \vec{u} &= g_1 \quad \textrm{in } \Omega\,,\\
	\vec{u} &= \vec{g_2} \quad \textrm{on } \del\Omega\,.
	\end{split}\end{equation}
	The boundary condition is satisfied in the trace sense. The solution comes with the apriori estimate
	\begin{equation*}
	\norm{\vec{u}}_{1,p} + \norm{\pi}_p \leq C \big (\norm{\vec{f}}_{-1,p} + \norm{g_1}_p + \norm{\vec{g_2}}_{1-\tfrac{1}{p},p}\big)\,.
	\end{equation*}
\end{prop}

The very weak solvability is defined as follows:
\begin{defi}
	We consider the space
	\begin{gather*}
	Y_{t'}(\Omega) := \{ \vec{\phi} \in W^{2,t'}(\Omega): \tr \vec{\phi} = \vec{0}, \tr \div \vec{\phi} = 0 \}
	\end{gather*}
	with the subspace norm. 
	
	A pair $(\vec{u}, \pi) \in L^t(\Omega) \times W^{-1,t}(\Omega)$ is a very weak solution of the Stokes problem~\eqref{eq:stokes_inhom} with $\vec{f}=\vec{0}$ if it holds
	\begin{align*}
	- \intom \vec{u} \cdot \Delta \vec{\phi} \dx - \fp{\pi}{\div \vec{\phi}}_{W_0^{1,t'}(\Omega)}
	+ \fp{\vec{g_2}-(\vec{g_2}\dotn) \vec{n}}{\del_n \vec{\phi}}_{W^{\tfrac{1}{t},t'}(\del\Omega)} &= 0\,,\\
	\intom \vec{u}  \cdot \nabla \psi \dx + \intom g_1 \psi \dx
	- \fp{\vec{g_2}\dotn}{\psi}_{W^{\tfrac{1}{t},t'}(\del\Omega)} &= 0
	\end{align*}
	for all $\vec{\phi} \in Y_{t'}(\Omega)$, $\psi \in W^{1,t'}(\Omega)$.
\end{defi}

The existence result is proved in \cite[Prop.~1 and~2, Thm.~11]{SON_Amrouche_2011} for $d=3$, but it is straight-forward transferable to $d=2$:
\begin{prop}
	Let $t \in (1, \infty)$.
	$(\vec{u}, \pi) \in L^t(\Omega) \times W^{-1,t}(\Omega)$ is a very weak solution of the Stokes problem~\eqref{eq:stokes_inhom} with $\vec{f}=\vec{0}$ if and only if it is a distributional solution.
	
	For $g_1 \in L^s(\Omega)$ and $\vec{g_2} \in W^{-\tfrac{1}{t},t}(\del\Omega)$ with $t = s\d$, there is exactly one very weak solution. It satisfies
	\begin{equation*}
	\norm{\vec{u}}_{t} + \norm{\pi}_{-1,t} \leq C \big (\norm{g_1}_s + \norm{\vec{g_2}}_{-\tfrac{1}{t},t}\big)\,.
	\end{equation*}
\end{prop}

Obviously, a weak solution is also distributional solution and, hence, a very weak solution. Due to the uniqueness of very weak solutions, we find that both solutions agree if $g_1 \in L^p(\Omega) \cap L^s(\Omega)$, $\vec{g_2} \in W^{1-\tfrac{1}{p},p}(\del\Omega) \cap W^{-\tfrac{1}{t},t}(\del\Omega)$ and $t = s\d$.

\subsection{Solution for Lipschitz domains}

We assume $\del \Omega \in C^{0,1}$. Concerning the weak solvability of the Dirichlet Laplace problem, we have the following result (this follows from \cite[Thm.~0.5]{Lipschitz_JK95} in combination with a standard trace lift, see also the discussion in~\cite[p.~2]{Laplace_Amrouche22}):
\begin{prop}\label{prop:ex_weak_laplace}
	Let $p \in [\tfrac{2d}{d+1}, \tfrac{2d}{d-1}]$ and $g_2 \in W^{1-\tfrac{1}{p},p}(\del\Omega)$. Then there is a unique weak solution $u \in W^{1,p}(\Omega)$ to
	\begin{align} \label{eq:fs:laplace}
	\begin{aligned}
	- \Delta u &= 0 && \textrm{in } \Omega\,,\\
	u &= g_2 && \textrm{on } \del\Omega\,.
	\end{aligned}
	\end{align}
	It satisfies
	\begin{equation*}
	\norm{u}_{1,p} \leq C \norm{g_2}_{1-\tfrac{1}{p},p}.
	\end{equation*}
\end{prop}

\begin{cor}\label{cor:fs:laplace_lq}
	Let $p \in [\tfrac{2d}{d+1}, \tfrac{2d}{d-1}]$, $t \geq 2$ and $g_2 \in W^{1-\tfrac{1}{p},p}(\del\Omega) \cap L^t(\del\Omega)$. Then the weak solution $u \in W^{1,p}(\Omega)$ of \eqref{eq:fs:laplace} lies also in $L^t(\Omega)$ and fulfills
	\begin{equation*}
	\norm{u}_{t} \leq C \norm{g_2}_{t}\,.
	\end{equation*}
\end{cor}
\begin{proof}
	We want to apply \cite[Thm.~1]{Dahlberg1979}, where the
	$L^t$-estimate is provided for the Poisson integral of $g_2$. Due to
	\cite[p.~2]{Dahlberg1979}, this statement, which is formulated for
	$d \geq 3$, holds for $d = 2$ as well.  To begin with, let in
	addition $g_2$ be a continuous function. We notice that Condition
	(i) in \cite[Thm.~1]{Dahlberg1979} is fulfilled for any bounded
	Lipschitz domain $\Omega$. Hence, $g_2$ is $L^t$-integrable
	w.r.t.~the harmonic measure $\omega$, i.e.
	$\tilde{u} := \int_{\del\Omega} g_2 \omega \in L^t(\Omega)$ with
	$\norm{\tilde{u}}_{t} \leq C \norm{g_2}_{t}$.
	
	It follows from \cite[Thm.~4.4b]{Nonsmooth_JK81} that $\tilde{u}$ is harmonic in $\Omega$. Furthermore, $\Omega$ is non-tangentially accessible and, hence, regular for the Dirichlet problem~\cite[p.~14]{NTA_JK82}, so $\tilde{u} \in C(\closed{\Omega})$ and $\tr \tilde{u} = \tilde{u} = g_2$ on $\del\Omega$. Thus, we have $u = \tilde{u}$ due to the uniqueness of $u$, cf. Proposition~\ref{prop:ex_weak_laplace}.
	The result then follows from \cite[Thm.~1]{Dahlberg1979} for continuous $g_2$.
	
	For general $g_2$, we note that $C(\del\Omega) \cap W^{1-\tfrac{1}{p},p}(\del\Omega) \cap L^t(\del\Omega)$ is a dense subset of $W^{1-\tfrac{1}{p},p}(\del\Omega) \cap L^t(\del\Omega)$ because of a classical convolution argument~\cite{SobolevSpaces_Adams03}. Thus, the claim follows by approximation.
\end{proof}

\begin{rem}
	Refined estimates are possible in some special cases, e.g. if the domain is simply-connected \cite{Shen95_Stokes}. We refer to \cite{Kenig85_Recent} for an overview on results on very weak solutions of elliptic boundary value problems on Lipschitz domains.
\end{rem}

A component-wise application of Corollary \ref{cor:fs:laplace_lq}
proves Proposition \ref{prop:s:fs_lipschitz}.

\printbibliography

\end{document}